\documentclass[microtype]{gtpart}
\usepackage{graphicx}
\usepackage[mathscr]{eucal}
\usepackage{amssymb}
\usepackage{xcolor}
\definecolor{darkred}{RGB}{200, 60, 0}
\definecolor{mildblue}{RGB}{0, 100, 250}
\usepackage{enumerate, cite}
\usepackage[margin=1.25in]{geometry}
\usepackage{hyperref}
\hypersetup{
    colorlinks=true,
    linkcolor=darkred,
    urlcolor=darkred,
    citecolor=mildblue,      
    urlcolor=darkred,
}
\usepackage[nameinlink]{cleveref}
\usepackage[nottoc]{tocbibind}
\newcommand{\br}{\mathbb{R}}

\newcommand{\bz}{\mathbb Z}
\newcommand{\bn}{\mathbb N}

\newcommand{\vp}{\varphi}

\newcommand{\ssm}{\smallsetminus}

\DeclareMathOperator{\mcg}{MCG}

\DeclareMathOperator{\Homeo}{Homeo}

\DeclareMathOperator{\supp}{supp}

\DeclareMathOperator{\rank}{rank}

\renewcommand{\co}{\colon\thinspace}

\newtheorem{theorem}{Theorem}[section]
\newtheorem{proposition}[theorem]{Proposition}
\newtheorem{lemma}[theorem]{Lemma}
\newtheorem{corollary}[theorem]{Corollary}

\newtheorem*{maintheorem1}{\Cref{thm:closed manifold}}
\newtheorem*{maintheorem2}{\Cref{thm:telescoping manifolds}}
\newtheorem*{maintheorem3}{\Cref{thm:well-ordered}}
\newtheorem*{maincorollary3}{\Cref{cor:proper class}}

\theoremstyle{definition}
\newtheorem{Def}[theorem]{Definition}

\newtheorem*{example}{Example}
\newtheorem{remark}[theorem]{Remark}

\numberwithin{equation}{section}

\title{Strongly bounded generation in transformation groups}

\author{Nicholas G. Vlamis}
\address{Department of Mathematics \\ CUNY Graduate Center \\ New York, NY 10016, and \newline Department of Mathematics \\ CUNY Queens College \\ Flushing, NY 11367}
\email{nvlamis@gc.cuny.edu}

\begin{document}  

\begin{abstract}
	Word metrics on finitely generated groups have canonical quasi-isometry classes, making quasi-isometry invariants genuine group invariants. 
	Rosendal generalized this phenomenon to topological groups through CB-generation, but in the general topological setting the resulting quasi-isometry invariants are not invariants of the underlying abstract group. 
	Specializing to the discrete case yields what we call SB-generated groups, where the invariants are genuinely algebraic. 
	We show that SB-generation arises naturally in transformation groups  by identifying several broad families of examples: the identity component of homeomorphism groups of closed manifolds, certain big mapping class groups, and homeomorphism groups of compact well-ordered spaces with successor limit capacity. 
	These results demonstrate that SB-generation provides a robust extension of finite generation.
\end{abstract}

\maketitle

\vspace{-0.2in}

\makeatletter
\let\origtableofcontents\tableofcontents
\renewcommand{\tableofcontents}{%
  \begingroup
    \setlength{\parskip}{0pt}% tighten vertical gaps
    \setlength{\parindent}{0pt}%
    \origtableofcontents % call the original (keeps the title)
  \endgroup
}
\makeatother
\setcounter{tocdepth}{1}
\tableofcontents

%-------------------
% The Basics
%-------------------

\section{Introduction}

A central tenet of geometric group theory is to regard a group as a geometric object.
To do so, one endows the group with a metric relevant to its group-theoretic structure; for instance, a standard requirement is for the metric to be left (or right) invariant. 
Having chosen such a metric, the next question is whether it is canonical in some sense---so that its large-scale geometry encodes invariants of the group itself.

For finitely generated groups, the canonical choice is the word metric associated to a finite generating set, since any two such word metrics are quasi-isometric (indeed, bi-Lipschitz equivalent).
In the setting of topological groups, compactly generated groups provide the natural analogue: word metrics associated to any two compact generating sets are quasi-isometric.
However, quasi-isometry invariants of these word metrics are not, in general, isomorphism invariants of the underlying abstract group.
For instance, consider the following groups: the circle group \( \mathbb T \), the complex numbers under multiplication \( \mathbb C^\times \), \( \br \), and \( \br^2 \), where the latter two are equipped with standard addition. 
In each of these groups, word metrics arising from compact generating sets are quasi-isometric to the group’s standard geometry.
Yet, as abstract groups, \( \mathbb T \) and \( \mathbb C^\times \) (respectively, \( \br \) and \( \br^2 \)) are isomorphic but not quasi-isometric.
This illustrates  that quasi-isometry invariants of such word metrics are not invariants of the underlying abstract groups.

Our motivation is to exhibit examples of groups that admit generating sets for which the associated word metrics all lie in a single quasi-isometry class canonically determined by the group, so that quasi-isometry invariants become genuine isomorphism invariants (this is made precise in \Cref{rem:motivation}). 
Rosendal \cite{RosendalCoarse} showed that the groups with this property are precisely the discrete CB-generated groups.
Recall that a subset of a topological group is \emph{coarsely bounded} if it has finite diameter in every left-invariant continuous metric on the group, and a topological group is \emph{CB-generated} if it is generated by a coarsely bounded subset.

In this paper, we establish the existence of rich, natural families of discrete CB-generated groups, demonstrating the value of restricting Rosendal's work to the setting of abstract groups. 
To emphasize the absence of topology, we introduce the following terminology.

\begin{Def}[SB-generated group]
	A subset of a group \( G \) is \emph{strongly bounded} if it has finite diameter in every left-invariant metric on \( G \).
	A group is \emph{SB-generated} if it is generated by a strongly bounded subset. 
\end{Def}

Every finitely generated group is SB-generated, but the class is strictly larger.
A group is \emph{strongly bounded} if it is a strongly bounded subset of itself.
Note that every strongly bounded group is SB-generated. 
Many strongly bounded groups exist in the literature: e.g., the homeomorphism group of the \( n \)-sphere \cite{CalegariDistortion}, the homeomorphism group of \( \br^n \) \cite{MannLargescale}, the symmetric group on the natural numbers \cite{BergmanGenerating}, homeomorphism groups of telescoping surfaces \cite{VlamisHomeomorphisma,VlamisHomeomorphism}, and homeomorphism groups of well-ordered spaces of Cantor--Bendixson degree one and with successor limit capacity \cite{BhatAlgebraic}. 

Given the existence of strongly bounded groups, simple constructions establish the existence of non-strongly bounded SB-generated groups---for example, the free product of two strongly bounded groups.
Here, our main results provide natural, structurally rich families of SB-generated groups that are not finitely generated SB-generated groups and that include non-strongly bounded groups.

\subsection*{Homeomorphism groups of closed manifolds}

Let \( M \) be a  manifold, and let \( \Homeo(M) \) denote the group of homeomorphisms \( M \to M \).
Denote by \( \Homeo_0(M) \)  the connected component of the identity, when equipped with the compact-open topology.

\begin{maintheorem1} 
	If \( M \) is a closed manifold, then \( \Homeo_0(M) \) is SB-generated. 
\end{maintheorem1}

This strengthens a theorem of Mann--Rosendal \cite{MannLargescale}, which states that \( \Homeo_0(M) \) is CB-generated in the compact-open topology. 
To prove \Cref{thm:closed manifold}, we show that the coarsely bounded generating set given by Mann--Rosendal is in fact strongly bounded.  
Mann--Rosendal also showed  that if the dimension of \( M \) is at least two and \( \pi_1(M) \) has an infinite order element, then \( \Homeo_0(M) \) is not strongly bounded. 

Recall that the \emph{mapping class group} of a manifold \( M \) (also known as the \emph{homeotopy group}), denoted \( \mcg(M) \), is the group of isotopy classes of homeomorphisms \( M \to M \). 

\begin{corollary} 
\label{cor:mcg}
	Let \( M \) be a closed manifold.
	If \( \mcg(M) \) is finitely generated, then \( \Homeo(M) \) is SB-generated.
	\qed
\end{corollary}

In dimension two this applies to every closed surface, since surface mapping class groups are finitely generated, a theorem of Dehn (see \cite[Chapter~4]{FarbPrimer}).

\subsection*{Big mapping class groups}

Our next family arises from big mapping class groups.
A  surface is of \emph{finite type} if its interior is homeomorphic to the interior of a compact surface; otherwise, it is of \emph{infinite type}.
A mapping class group of a surface is called \emph{big} if the  surface is of infinite type.
For a surface \( S \), the \emph{compact-open topology} on \( \mcg(S) \) is the quotient topology inherited from the compact-open topology on \( \Homeo(S) \).

%We exhibit a family of big mapping class groups that are SB-generated. 
Mann--Rafi \cite[Theorems~1.6~\&~1.7]{MannLargescalea} classified the tame surfaces whose mapping class groups are CB-generated when equipped with the compact-open topology. 
In \Cref{cor:telescoping manifolds}, we give a subclass of these surfaces whose mapping class groups are SB-generated.
%Rather than introduce the topological conditions that guarantee CB-generation, in our statement we will simply require the surfaces under consideration to have CB-generated mapping class groups.  
The description of this subclass relies on the notion of a telescoping surface,  introduced by Mann--Rafi in \cite{MannLargescalea} and expanded by the author in \cite{VlamisHomeomorphism}. 
Telescoping surfaces naturally partition into three types based on the cardinality of their set of maximal ends (see \Cref{sec:mcg}\label{fn:tame} for definitions).
Informally, each type reflects the homogeneity of one of the following surfaces:  the 2-sphere (perfect set of maximal ends), the plane (a unique maximal end), or the open annulus (two maximal ends). 
A telescoping surface with a unique maximal end is called \emph{uniquely telescoping}.

Given a subsurface \( \Sigma \) of a surface \( S \), we let \( U_\Sigma \subset \mcg(S) \) denote the mapping classes admitting a representative homeomorphism that restricts to the identity on \( \Sigma \). 
When \( \Sigma \) is finite type, the set \( U_\Sigma \) is a clopen neighborhood of the identity in the compact-open topology on \( \mcg(S) \).
The main theorem (\Cref{thm:telescoping manifolds}) provides a condition guaranteeing the existence of a finite-type subsurface \( \Sigma \) of \( S \) such that \( U_\Sigma \) is strongly bounded, or in other words, for when \( \mcg(S) \) is \emph{locally strongly bounded}. 

% The theorem requires that the telescoping surfaces are \emph{tame}, a technical condition introduced by Mann--Rafi in their classification of CB-generated big mapping class groups.  
% Constructing a non-tame surface is nontrivial. 
% We provide definitions and background in \Cref{sec:mcg}. 

\begin{maintheorem2}
	Let \( S \) be a surface that can be expressed as the connected sum of a finite-type borderless surface and finitely many  telescoping surfaces \( M_1, \ldots, M_n \).
	If, for each \( M_i \) that is uniquely telescoping, there exists \( j \neq i \) with \( M_j \) homeomorphic to \( M_i \), then there exists a finite-type subsurface \( \Sigma \) of \( S \) such that \( U_\Sigma \) is strongly bounded\footnote{In the actual \Cref{thm:telescoping manifolds} below, \( U_\Sigma \) is shown to be strongly distorted, a stronger condition. See \Cref{section:basics}.} in \( \mcg(S) \).
\end{maintheorem2}

% \begin{maintheorem2}
% 	Let \( S \) be a surface that can be expressed as the connected sum of a finite-type borderless surface and finitely many tame telescoping surfaces \( M_1, \ldots, M_n \).
% 	If, for each \( M_i \) that is uniquely telescoping, there exists \( j \neq i \) with \( M_j \) homeomorphic to \( M_i \), then \( \mcg(S) \) is SB-generated.
% \end{maintheorem2}

In \Cref{cor:cb-to-sb}, we show that every strongly locally bounded CB-generated Polish group is SB-generated.
Mapping class groups are Polish groups, and therefore, the following corollary is an immediate consequence of \Cref{thm:telescoping manifolds} and \Cref{cor:cb-to-sb}.

% Now, if \( S \) are as in \Cref{thm:telescoping manifolds} and \( \mcg(S) \) is CB-generated, then the CB-generating set given by Mann--Rafi \cite[Section~6]{MannLargescalea} is of the form \( U_\Sigma \cup F \), where \( \Sigma \) is the subsurface given by \Cref{thm:telescoping manifolds} and \( F \) is a finite set. 
% In particular, \( U_\Sigma \cup F \) is a strongly bounded generating set for \( \mcg(S) \), yielding the following corollary.

\begin{corollary} 
	\label{cor:telescoping manifolds}
	Let \( S \), \( M_1, \ldots, M_n \) be as in \Cref{thm:telescoping manifolds} and suppose that \( \mcg(S) \) is CB-generated with respect to the compact-open topology.
 	If, for each \( M_i \) that is uniquely telescoping, there exists \( j \neq i \) with \( M_j \) homeomorphic to \( M_i \), then \( \mcg(S) \) is SB-generated.
	\qed
\end{corollary}

\Cref{cor:telescoping manifolds} does not account for all SB-generated mapping class groups.
For instance, if \( S \) is a uniquely telescoping surface, then \( \mcg(S) \) is strongly bounded \cite{VlamisHomeomorphism} and hence SB-generated; however, it does not satisfy the hypotheses of the corollary. 
With this exception, it is natural to ask if the converse of \Cref{cor:telescoping manifolds} is true.

As noted above, Mann--Rafi gave a topological classification of the tame surfaces whose mapping class groups are CB-generated.
Therefore, in \Cref{cor:telescoping manifolds}, the requirement that \( \mcg(S) \) is CB-generated can be replaced with topological conditions on \( S \).
As these topological conditions, as well as the notion of a tame surface, are technical to state, we believe it best to refer the interested reader to \cite{MannLargescalea} for details.

A subset \( \Sigma \) of a surface \( S \) is \emph{displaceable} if there exists a homeomorphism \( f \co S \to S \) such that \( f(\Sigma) \cap \Sigma = \varnothing \); otherwise, it is \emph{non-displaceable}.
It follows from Mann--Rafi's classification of coarsely bounded mapping class groups \cite[Theorem~1.7]{MannLargescalea} that if \( S \) is as in \Cref{cor:telescoping manifolds}, then \( \mcg(S) \) fails to be strongly bounded whenever \( S \) contains a compact non-displaceable  subset.
For example, if \( S \) is as in \Cref{cor:telescoping manifolds} and among \( \{M_1,\ldots, M_n\} \) there are two non-homeomorphic telescoping surfaces, or at least three uniquely telescoping surfaces, then \( \mcg(S) \) is not strongly bounded. 

\begin{example}
	Let \( K \subset \br^2 \) be an embedded copy of the Cantor set.
	\Cref{cor:telescoping manifolds} implies that \( \mcg(\br^2 \ssm K) \) is SB-generated.
	Moreover, with respect to the word metric associated to a strongly bounded generating set, \( \mcg(\br^2\ssm K) \) is an infinite-diameter Gromov hyperbolic group. 
	This was previously deduced in the topological group setting as follows:
	Mann--Rafi \cite{MannLargescalea} showed that \( \mcg(\br^2\ssm K) \) is CB-generated in the compact-open topology.
	With respect to the word metric associated to a coarsely bounded generating set, Schaffer-Cohen \cite{SchafferGraphs} showed that \( \mcg(\br^2 \ssm K) \) is quasi-isometric to the ray graph.
	Bavard \cite{BavardHyperbolicity} showed that the ray graph is Gromov hyperbolic and infinite diameter.
	\Cref{cor:telescoping manifolds} promotes the coarsely bounded generating set to a strongly bounded generating set, yielding the fact that \( \mcg(\br^2 \ssm K) \) is an infinite-diameter Gromov hyperbolic SB-generated group.
\end{example}

From the definition, it readily follows that any quotient of an SB-generated group is itself SB-generated.
Below, in \Cref{prop:abelian}, we show that an abelian group is SB-generated if and only if it is finitely generated.
As an application of these facts, the abelianization of an SB-generated group is finitely generated (\Cref{cor:abelianization}). 
Therefore, from \Cref{cor:telescoping manifolds} we recover a theorem of Field--Patel--Rasmussen \cite[Theorem~1.4]{FieldStable}; in fact, we broaden their theorem to a larger class of surfaces.

\begin{corollary}\label{cor:abelianization big mcg}
	If \( S \) is as in \Cref{cor:telescoping manifolds}, then the abelianization of \( \mcg(S) \) is finitely generated.
	\qed
\end{corollary}

\subsection*{Homeomorphism groups of well-ordered spaces}

Our final family arises from well-ordered spaces.
A \emph{well-ordered space} is a well-ordered set equipped with its order topology (see \Cref{sec:well-ordered} for definitions).
Up to homeomorphism, compact well-ordered sets are classified by two invariants, their limit capacity (an ordinal) and their degree (a natural number). 
We focus on the case where the limit capacity is a successor ordinal.

\begin{maintheorem3} 
	If the limit capacity of a compact well-ordered space is a successor ordinal, then its homeomorphism group is SB-generated.
\end{maintheorem3}

As a consequence, one obtains arbitrarily large families of non-finitely generated, non-strongly bounded SB-generated groups.

\begin{maincorollary3} 
	For any cardinal \( \kappa \), there exists a set of pairwise non-isomorphic, non-finitely generated, non-strongly bounded, SB-generated groups of cardinality \( \kappa \). 
\end{maincorollary3}

Countable well-ordered spaces can be realized as end spaces of surfaces, linking the theory of big mapping class groups with that of homeomorphism groups of well-ordered spaces; see \Cref{sec:well-ordered} for details.

\subsection*{SB-generated groups and unique Polish topologies}

Let \( G \) be an SB-generated group. 
It readily follows from basic properties of SB-generated groups (discussed in \Cref{section:basics}) that if \( G_1 \) and \( G_2 \) are topological groups abstractly isomorphic to \( G \), then both \( G_1 \) and \( G_2 \) are CB-generated and are quasi-isometric. 
The converse is more subtle. 
We conclude the introduction with exploring a version of this question in the setting of Polish groups.

A topological group is \emph{Polish} if its underlying topology is separable and completely metrizable. 
If a group admits a unique Polish group topology, then this topology is an isomorphism invariant of the group.
In particular, a natural place to look for SB-generated groups is among the CB-generated Polish groups in which the Polish group structure is unique.  
Under mild hypotheses, Kallman \cite{KallmanUniqueness} showed that homeomorphism groups of second-countable Hausdorff spaces  have a unique Polish topology; in particular, \( \Homeo_0(M) \) has a unique Polish group topology whenever \( M \) is a  manifold.

However, not every CB-generated Polish group with a unique Polish group topology is SB-generated.
In forthcoming joint work with T.~Ghaswala, S.~Iyer, and R.~Lyman, we show that all big mapping class group admit a unique Polish group topology. 
Yet, by Domat--Dickmann \cite{DomatBig}, some such groups surject onto \( \mathbb Q \), and hence cannot be SB-generated by \Cref{cor:abelianization}.
We record this fact below.

\begin{proposition}
	There exists a CB-generated topological group with a unique Polish group topology that is not SB-generated.
	\qed
\end{proposition}

\section*{Acknowledgments}

The author thanks Robbie Lyman and Jing Tao for helpful discussions and Sanghoon Kwak for the reference to the algebraic rigidity of homeomorphism groups of ordinals. 
The author was partially supported by NSF DMS-2212922 and PSC-CUNY Award 67380-00 55.

%-------------
% The basics
%-------------

\section{SB-generated groups}\label{section:basics}

	In this section, we provide the basic properties of SB-generated groups.
	Recall from the introduction that a subset of a group \( G \) is \emph{strongly bounded} if it has finite diameter in every left-invariant metric on \( G \), and \( G \) is \emph{SB-generated} if it is generated by a strongly bounded set.

	The class of SB-generated groups is a proper subclass of the CB-generated groups introduced by Rosendal \cite{RosendalCoarse} in the category of topological groups; in particular, an SB-generated group is a discrete CB-generated group. 
	As a consequence, the statements below---through \Cref{prop:maximal}---are special cases of results in \cite[Section~2]{RosendalCoarse}. 
	Despite this, we reproduce the results here in the language of abstract groups, with the advantage that the reader does not have to translate from the more general setting of topological groups.

	\begin{remark}
		We only discuss the results in \cite[Section~2]{RosendalCoarse} most relevant to the discussion on hand; however, we strongly encourage the reader interested in working with SB-generated groups to read the section with discrete groups in mind. 
		For instance, there is a version of the Milnor--\v Svarc lemma \cite[Theorem~2.77]{RosendalCoarse} in this setting that we do not introduce here. 
	\end{remark}

	In what follows, we will work with the following preorder on the left-invariant pseudo-metrics on a group \( G \):
	Given two left-invariant pseudo-metrics \( d_1 \) and \( d_2 \), we write \( d_1 \preceq d_2 \) if there exists \( K > 0 \) such that 
	\[ 
		d_1(g,h) \leq K \cdot d_2(g,h) + K
	\]
	for every \( g,h \in G \).
	Note that if \( d_1 \preceq d_2 \) and \( d_2 \preceq d_1 \), then \( d_1 \) and \( d_2 \) are quasi-isometric. 

	We will be interested in maximal metrics.
	There are two natural notions of maximal, one being that every metric is below a maximal metric and the other being that no metric is above a maximal metric.
	A simple observation shows that these notions agree here: given any two left-invariant metrics \( d_1 \) and \( d_2 \) on \( G \), we have \( d_1,d_2 \preceq d_1 + d_2 \).

	\begin{Def}
		A left-invariant pseudo-metric \( d \) on \( G \) is \emph{maximal} if \( \rho \preceq d \) for every left-invariant pseudo-metric \( \rho \) on \( G \). 
	\end{Def}

	By definition, any two maximal metrics on a group are quasi-isometric. 
	Consequently, if a group admits a maximal metric, then every quasi-isometry invariant of this metric is an isomorphism invariant of the group.

	Given a generating set \( S \) for a group \( G \), define the associated \emph{word norm} \( |\cdot |_S \co G \to \bn \cup \{0\} \) by \( |g|_S = \min\{ n \in \bn\cup\{0\} : g \in S^n \} \), where \( S^0 = \{1\} \) and \( S^n = \{ s_1s_2 \cdots s_n : s_i \in S\} \) when \( n \in \bn \).
	The associated \emph{word metric}, denoted \( d_S \), is given by \( d(g,h) = |h^{-1}g| \).
	Note that the word metric is left invariant. 

	\begin{lemma}
		Let \( S \) be a strongly bounded generating set for an SB-generated group \( G \).
		If \( \rho \) is a left-invariant pseudo-metric on \( G \), then the identity map from \( (G, d_S) \) to \( (G, \rho) \) is Lipschitz. 
	\end{lemma}

	\begin{proof}
		Let \( d = d_S \) be the word metric associated to \( S \).
		Let \( \rho \) be a pseudo-metric on \( G \). 
		As \( S \) is strongly bounded, there exists \( K > 0 \) such that \( \rho(1,s) \leq K \) for all \( s \in S \). 
		Given \( g, h \in G \), write \( h^{-1} g = s_1 s_2 \cdots s_n \), where \( s_i \in S \) and \( n = d(g,h) \). 
		Then
		\begin{align*}
			\rho(g,h)	&=		\rho(1,h^{-1}g) \\
					&=		\rho(1, s_1s_2 \cdots s_n) \\
					&\leq		\sum_{k=1}^n \rho(1, s_k) \\
					&\leq		K\cdot n \\
					&=		K\cdot d(g,h)
		\end{align*}
		The result follows as \( K \) is independent of \( g \) and \( h \). 
	\end{proof}

	As immediate corollaries, we see that any word metric associated to an SB-generating set is maximal and any two such word metrics are bi-Lipschitz equivalent. 

	\begin{corollary}
	\label{cor:maximal word}
		In an SB-generated group, the word metric associated to a strongly bounded generating set is maximal. 
		\qed
	\end{corollary}

	\begin{corollary}
		In an SB-generated group, the word metrics associated to any two strongly bounded generating sets are bi-Lipschitz equivalent. 
		\qed
	\end{corollary}

	Our next proposition, \Cref{prop:maximal}, tells us that a group admits a maximal metric if and only if it is SB-generated.

	\begin{proposition} 
	\label{prop:maximal}
		A left-invariant pseudo-metric on a group is maximal if and only if it is quasi-isometric to the word metric associated to a strongly bounded generating set. 
	\end{proposition}

	\begin{proof}
		Let \( G \) be a group and let \( d \) be a left-invariant pseudo-metric on \( G \). 
		First, assume that \( d \) is maximal. 
		For \( n \in \bn \cup \{ 0 \} \), let \( H_n \) be the subgroup generated by the set \( B_n =  \{ g \in G : d(1,g) \leq n \} \).  
		We claim that there exists \( n \in \bn \) such that \( G = H_n \). 
		If not, then we can define a left-invariant metric \( \rho \) by setting \( \rho(g,h) = \min \{ n^2  : n \in \bn \text{ and }  h^{-1}g \in H_n \} \). 
		Taking \( g_n \in H_n \ssm H_{n-1} \), we get a sequence such that \( \rho(1,g_n)/ d(1,g_n) \to \infty \) as \( n \to \infty \), implying that \( \rho \not\preceq d \) and therefore contradicting the maximality of \( d \). 
		Hence, \( G = H_n \) for some \( n \in \bn \). 
		The maximality of \( d \) then implies that \( B_n \) is a strongly bounded generating set for \( G \). 
		By \Cref{cor:maximal word}, the word metric associated to \( B_n \) is maximal, and as any two maximal metrics are quasi-isometric, \( d \) is quasi-isometric to the word metric associated to \( B_n \).

		Now, assume that \( d \) is quasi-isometric to the word metric associated to a strongly bounded generating set.
		Again by \Cref{cor:maximal word}, this word metric is maximal and hence so is \( d \).
	\end{proof}

	\begin{remark}
		\label{rem:motivation}
		\Cref{prop:maximal} captures the central motivation for studying SB-generated groups: 
		their word metrics are canonically associated to the group in the sense that they are maximal. 
		Consequently, any quasi-isometry invariant of such a word metric is an isomorphism invariant of the group.
	\end{remark}

	It is clear that every finitely generated group is SB-generated, allowing us to view SB-generated groups as a generalization of finitely generated groups.
	Serre, in studying actions on trees \cite{SerreArbres}, introduced the notion of uncountable cofinality as a generalization of finite generation.
	It will be helpful for us to see that SB-generated groups also have this property.

	A group has \emph{uncountable cofinality} if it cannot be expressed as the union of a strictly increasing sequence of subgroups. 
	It is more or less apparent from the proof of \Cref{prop:maximal} that an SB-generated group has uncountable cofinality; nonetheless, we give an argument.

	\begin{proposition}\label{prop:cofinal}
		SB-generated groups have uncountable cofinality.
	\end{proposition}

	\begin{proof}
		Let \( G \) be an SB-generated group, and let \( S \) be a strongly bounded generating set. 
		Given a sequence of subgroups \( \{H_n\}_{n\in\bn} \) such that \( G = \bigcup_{n\in\bn} H_n \), define \( \rho(g,h) = \min\{n \in \bn : h^{-1}g \in H_n \} \). 
		Then \( \rho \) is a left-invariant metric, and as \( S \) is strongly bounded, the \( \rho \)-diameter of \( S \) is bounded, implying that \( S \subset H_n \) for some \( n \in \bn \); in particular, \( G = H_n \). 
	\end{proof}

	It is an exercise to check that when restricted to countable groups, the notion of finitely generated, SB-generated, and having uncountable cofinality are all equivalent.
	And, as we will now argue, these notions also all agree in the setting of  abelian groups.

	\begin{proposition}\label{prop:abelian}
		If \( A \) is an abelian group, then the following are equivalent:
		\begin{enumerate}[(i)]
			\item \( A \) is finitely generated.
			\item \( A \) is SB-generated.
			\item \( A \) has uncountable cofinality.
		\end{enumerate}
	\end{proposition}

	\begin{proof}
		By \Cref{prop:cofinal}, we need only show that an abelian group with uncountable cofinality is finitely generated. 
		Let \( A \) be an abelian group that cannot be finitely generated. 
		Using structure theorems for abelian groups, specifically \cite[Theorem~23.1 and Theorem~24.1]{FuchsInfinite}, we can realize \( A \) as a subgroup of a group of the form \( \oplus_{i \in I} A_i \), where each \( A_i \) is either quasicyclic or isomorphic to the rationals. 
		By forgetting terms in the summand, we may assume that the natural projection \( A \to A_i \) is nontrivial for each \( i \in I \). 
		We have two cases: either the cardinality of \( I \) is finite or infinite.

		If the cardinality of \( I \) is infinite, then by choosing a denumerable subset \( J \) of \( I \), we obtain a homomorphism \( A \to \oplus_{j\in J} A_j \) such that the composition with the natural projection to \( A_j \) is nontrivial. 
		Choosing an enumeration of \( J = \{ j_n \}_{n\in \mathbb N} \) and setting \( B_n = \oplus_{i=1}^n A_{j_i} \), we see that \( \oplus_{j\in J} A_j \) fails to have uncountable cofinality, as \( \oplus_{j\in J} A_j = \bigcup_{n\in \bn} B_n \).
		Pulling back the \( B_n \) to \( A \), we see that \( A \) does not have uncountable cofinality.

		Now suppose that \( I \) is finite. 
		If the projection of \( A\) to each of the \( A_i \) is finitely generated, then \( A \) is contained in a finitely generated subgroup of \( \oplus_{i \in I} A_i \).
		But every subgroup of a finitely generated abelian group is finitely generated; hence, we can conclude  there exists \( i \in I \) such that the image of the projection of \(  A \) to \( A_i \) is not finitely generated. 
		Therefore, as a non-finitely generated countable group, the image of \( A \) in \( A_i \) fails to have uncountable cofinality, and hence so does \(  A \).
	\end{proof}

	Any image of an SB-generated group under a homomorphism is itself SB-generated.
	Therefore, as a corollary, we have that the abelianization of an SB-generated group is necessarily finitely generated, a useful tool for establishing that a group fails to be SB-generated.

	\begin{corollary} 
	\label{cor:abelianization}
		The abelianization of an SB-generated group is finitely generated.
		\qed
	\end{corollary}

	\subsection{Locally strongly bounded Polish groups}

	We now give a condition on a CB-generated Polish group that guarantees it is SB-generated.
	Recall that a topological group is \emph{Polish} if, as a topological space, it is separable and completely metrizable. 

	\begin{Def}
		A topological group is \emph{locally strongly bounded} if it admits an open neighborhood of the identity that is strongly bounded.
	\end{Def}

	The promotion of CB-generation to SB-generation in locally strongly bounded Polish groups will readily follow from the following proposition of Rosendal.

	\begin{proposition}[{\cite[Proposition~2.15]{RosendalCoarse}}]
		\label{prop:cb in polish}
		Let \( G \) be a Polish group.
		A subset \( A \) of \( G \) is coarsely bounded if and only if for every open neighborhood \( V \) of the identity in \( G \) there exists a finite set \( F \subset G \) and \( k \in \bn \) such that \( A \subset (FV)^k \).
		\qed
	\end{proposition}

	Before getting to the main proposition, we need a quick lemma stating that the product of two strongly bounded sets is strongly bounded.

	\begin{lemma}\label{lem:product of sb is sb}
		If \( A \) and \( B \) are strongly bounded sets in a group, then the product \[ AB = \{ab : a\in A, b \in B\} \] is strongly bounded.
	\end{lemma}

	\begin{proof}
		Let \( d \) be a left-invariant metric on \( G \).
		As \( A \) and \( B \) are strongly bounded, there exists \( D > 0 \) such that \( d(1,a), d(1,b) < D \) for all \( a \in A \) and \( b \in B \).
		Note that left-invariance implies \( d(1,g) = d(1,g^{-1}) \) for all \( g \in G \). 
		For \( i \in \{1,2\} \), let \( a_i \in A \) and \( b_i \in B \). 
		We want to bound \( d(a_1b_1, a_2b_2) \). 
		This is accomplished as follows:
		\begin{align*}
			d(a_1b_1,a_2b_2)	&= d(1,b_1^{-1}a_1^{-1}a_2b_2) \\
												&\leq d(1,b_1) + d(1, a_1) + d(1,a_2) + d(1,b_2) \\
												&\leq 4D.
		\end{align*}
		Hence, \( AB \) is strongly bounded. 
	\end{proof}

	\begin{proposition} \label{prop:cb-to-sb}
		Let \( G \) be a Polish group.
		If \( G \) is locally strongly bounded, then every coarsely bounded subset of \( G \) is strongly bounded.
	\end{proposition}

	\begin{proof}
		Let \( A \subset G \) be coarsely bounded. 
		As \( G \) is locally strongly bounded, there exists an open neighborhood of the identity, call it \( V \), that is strongly bounded.
		By \Cref{prop:cb in polish}, there exists a finite set \( F \subset G \) and \( k \in \bn \) such that \( A \subset (FV)^k \). 
		As \( (FV)^k \) is a product of strongly bounded groups, \Cref{lem:product of sb is sb} implies it is strongly bounded.
		Therefore, as \( A \) is a subset of a strongly bounded set, it must be strongly bounded itself.
	\end{proof}

	Applying \Cref{prop:cb-to-sb} to a CB-generating set yields the following corollary.

	\begin{corollary} 
	\label{cor:cb-to-sb}
		Every locally strongly bounded CB-generated Polish group is SB-generated.
		\qed
	\end{corollary}

%-------------------
% Local strong distortion
%-------------------

\section{Local strong distortion}\label{section:distortion}
	
	In this section, we introduce a technique for certifying that a subset of a group is strongly bounded. 
	Calegari--Freedman \cite{CalegariDistortion} introduced the notion of strong distortion for a group, and in the appendix of the same paper, Cornulier showed that a group with strong distortion is strongly bounded. 
	Here, we restrict the notion of strong distortion to subsets and reproduce Cornulier's argument in this setting to show that strongly distorted subsets are strongly bounded. 

	\begin{Def}[Strong distortion for subsets]
		A subset \( A \) of a group \( G \) is \emph{strongly distorted} if there exist \( m \in \bn \) and \( \{w_n\}_{n\in\bn} \subset \bn \) such that for every sequence \( \{a_n\}_{n\in\bn} \subset A \) there exists a subset \( S \) of \( G \) of cardinality \( m \) satisfying \( a_n \in S^{w_n} \). 
	\end{Def}

	The argument below showing that a strongly distorted subset is strongly bounded does not rely on the uniform constant \( m \), so we introduce the following weaker condition in case it is useful.
	Though in the latter sections, we will be able to produce uniform constants. 

	\begin{Def}
		A subset \( A \) of a group \( G \) is \emph{sequentially distorted} if there exists \( \{w_n\}_{n\in\bn} \subset \bn \) such that for every sequence \( \{a_n\}_{n\in\bn} \subset A \) there exists a finite subset \( S \) of \( G \) satisfying \( a_n \in S^{w_n} \). 
	\end{Def}

	\begin{lemma}\label{lem:distorted implies bounded}
		Every sequentially distorted subset of a group is strongly bounded. 
	\end{lemma}

	\begin{proof}
		Let \( A \) be a sequentially distorted subset of a group \( G \), and let \( \{ w_n \}_{n\in\bn} \subset \bn \) be the associated sequence. 
		As finite subsets are always strongly distorted and strongly bounded, we may assume that \( A \) is infinite, so that \( w_n \to \infty \) as \( n \to \infty \).   
		Fix  a left-invariant metric \( d \) on \( G \). 
		Suppose \( A \) has infinite \( d \)-diameter, allowing us to choose a sequence \( \{ a_n \}_{n\in\bn} \subset A \) such that \( d(1,a_n) \geq w_n^2 \). 
		By assumption, there exists a finite set \( S \) such that \( a_n \in S^{w_n} \). 
		As \( S \) is finite, we can define \( K = \max\{ d(1, s) : s \in S \} \). 
		It follows that \( d(1, a_n) \leq K\cdot w_n \) for all \( n \in \bn \), contradicting---for large \( n \)---the assumption that \( d(1, a_n) \geq w_n^2 \).
		Therefore, the \( d \)-diameter of \( A \) is bounded.
	\end{proof}

	In each of the families of SB-generated groups we establish below, we will rely on a fragmentation result.
	As a consequence, the generating sets we consider will have the form \( A_1A_2 \cdots A_n \) for some finite collection of subsets \( A_1, \ldots, A_n \).
	We will then show each of the \( A_i \) is strongly distorted.
	The final lemma of this section allows us to promote the strong distortion of each of the \( A_i \) to strong distortion of the product \( A_1A_2\cdots A_n \).

	\begin{lemma}
		\label{lem:product of distorted is distorted}
		If \( A_1, \ldots, A_n \) are sequentially distorted subsets of a group \( G \), then \[ A := A_1A_2 \cdots A_n = \{ a_1a_2\cdots a_n : a_i \in A_i \} \] is sequentially distorted in \( G \).
		Moreover, if each of the \( A_j \) is strongly distorted, then \( A \) is strongly distorted.
	\end{lemma}

	\begin{proof}
		Let \( \{a_k\}_{k\in\bn} \) be a sequence in \( A \).
		Then there exists \( a_{j,k} \in A_j \) such that \( a_k = a_{1,k}\cdots a_{n,k} \).
		Each of the \( a_{j,k} \) can be written as a word of length \( w_{j,k} \) in a set of cardinality of \( m_k \), where the \( w_{j,k} \) are independent of the initial sequence. 
		Therefore, \( a_n \) can be written as a word of length \( w_n = w_{1,k} + \cdots + w_{n,k} \) in a set of cardinality \( m_1+\cdots+m_k \), implying that \( A \) is sequentially distorted. 
		Moreover, if each of the \( A_j \) were strongly distorted, then the \( m_j \) can be chosen independent of the initial sequence, allowing us to conclude that \( A \) is strongly distorted. 
	\end{proof}

%-------------------
% Homeomorphism groups of closed manifolds
%-------------------

\section{Homeomorphism groups of closed manifolds}

	In this section we prove \Cref{thm:closed manifold}, showing that \( \Homeo_0(M) \) is SB-generated when \( M \) is a closed manifold. 
	Recall that \( \Homeo(M) \) is the group of homeomorphisms \( M \to M \) and is a topological group when equipped with the compact-open topology.
	We let \( \Homeo_0(M) \) denote the connected component of the identity. 

	The idea of the proof is to reduce---using the fragmentation lemma (\Cref{thm:fragmentation}) and \Cref{lem:product of distorted is distorted}---to showing that the subgroup of homeomorphisms supported in a standard open ball is strongly bounded.
	To do so, we show that the subgroup is strongly distorted using a standard technique  from the literature (\Cref{lem:construction}) and then apply \Cref{lem:distorted implies bounded}.

	Recall that the \emph{support} of a homeomorphism \( f\co M \to M \), denoted \( \supp(f) \), is the closure of the set \( \{ x \in M : f(x) \neq x \} \). 
	A subset \( B \) of an \( n \)-manifold is a \emph{standard open ball} if it is an open subset whose closure is homeomorphic to the closed unit ball in \( \br^n \) and whose boundary is a locally flat \( (n-1) \)-sphere. 
	Given two elements \( g \) and \( h \) in a group, we set \( g^h := h^{-1}gh \).

	We first establish that the subgroup of homeomorphisms supported in a given standard open ball is strongly distorted. 
	The proof relies on the following lemma.

	\begin{lemma}[{\cite[Construction~2.3]{LeRouxStrong}}]
		\label{lem:construction}
		Let \( X \) be a metric space. 
		Suppose \( Z \) is a subset of \( X \) such that there exist homeomorphisms \( \sigma, \tau : X \to X \) satisfying 
			\begin{enumerate}
				\item \( \sigma^n(Z) \cap \sigma^m(Z) = \varnothing \) and \( \tau^n(\supp(\sigma)) \cap \tau^m(\supp(\sigma)) = \varnothing \) for distinct \( n, m \in \bn \cup\{0\} \),
				\item if \( Y \) is a connected component of \( Z \), then the diameter of \( \sigma^n(Y) \) converges to 0 as \( n \to \infty \), and 				
				\item if \( Y \) is a connected component of \( \supp(\sigma) \), then the diameter of \( \tau^n(Y) \) converges to 0 as \( n \to \infty \). 
			\end{enumerate}
		If \( \{g_n\}_{n\in\bn} \) is a sequence of homeomorphisms supported in \( Z \), then 
			\[
				\gamma := \prod_{n,m \geq 0} g_n^{\tau^{-n}\sigma^{-m}}
			\]
		is a homeomorphism \( X \to X \) and 
			\[
				g_n = [\gamma^{\tau^{n}}, \sigma].
			\]
		\qed
	\end{lemma}

	\begin{proposition}
		\label{prop:construction}
		If \( X \), \( Z \), \( \sigma \), and \( \tau \) are as in \Cref{lem:construction}, then the subgroup \( A \) of \( \Homeo(X) \) consisting of homeomorphisms  supported in \( Z \) is strongly distorted in any closed subgroup of \( \Homeo(X) \) containing \( A \), \( \sigma \), and \( \tau \). 
	\end{proposition}

	\begin{proof}
		Let \( H \) be a closed subgroup of \( \Homeo(X) \) containing \( A \), \( \sigma \), and \( \tau \).
		Fix a sequence \( \{a_n\}_{n\in\bn} \) of \( A \). 
		Then \Cref{lem:construction} yields a homeomorphism \( \gamma \) such that \( a_n = [\gamma^{\tau^{n}}, \sigma] \); moreover, \( \gamma \in H \), as \( H \) is closed and \( \gamma \) is a limit of elements in \( H \). 
		We have therefore expressed \( a_n \) as a word of length \( 4n+4 \) in the set \( \{ \gamma^\pm, \tau^\pm, \sigma^\pm \} \).
		As the word length of \( a_n \) and the cardinality of the generating set were independent of the sequence, \( A \) is strongly distorted in \( H \).
	\end{proof}

	We now argue that by letting \( Z \) be a standard open ball in a manifold \( M \), we can construct \( \sigma \) and \( \tau \) satisfying the hypotheses of \Cref{lem:construction}. 
	Let \( {\mathbb B}^n \) denote the open unit ball in \( \br^n \).

	\begin{figure}[t]
		\centering
		\includegraphics{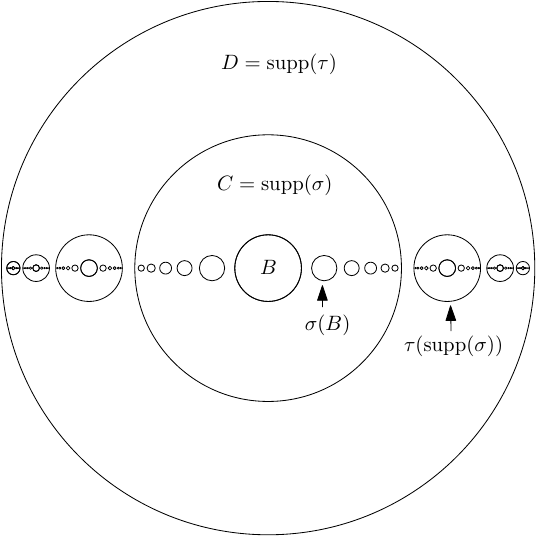}
		\caption{The homeomorphisms constructed in \Cref{lem:ball-distortion}.}
		\label{fig:manifolds}
	\end{figure}

	\begin{lemma}
		\label{lem:ball-distortion}
		If \( B \) is a standard open ball in \( M \), then the subgroup of \( \Homeo_0(M) \) consisting of homeomorphisms supported in \( B \) is strongly distorted in \( \Homeo_0(M) \). 
	\end{lemma}

	\begin{proof}
		By \Cref{prop:construction}, it is enough to find \( \sigma, \tau \in \Homeo_0(M) \) satisfying: 
		\begin{enumerate}[(i)]
			\item \( \sigma^n(B) \cap \sigma^m(B) = \varnothing \) for distinct \( n,m \in \mathbb Z \), 
			\item the diameter of \( \sigma^n(B) \) tends to 0 as \( n \) tends to \( \pm \infty \),
			\item \( \tau^n(\supp(\sigma)) \cap \tau^m(\supp(\sigma)) = \varnothing \) for distinct \( n, m\in \mathbb Z \), and
			\item the diameter of \( \tau^n(\supp(\sigma)) \) tends to 0 as \( n \) tends to \( \pm \infty \).
		\end{enumerate}

		As \( B \) is a standard open ball, we can choose another standard open ball \( C \) containing the closure of \( B \).
		Fix a homeomorphism \( \varphi \co C \to \mathbb B^{n-1} \times \mathbb R \)		sending \( B \) to the open ball of radius 1/2 in \( \mathbb B^{n-1} \times \mathbb R \) (viewing this as a subspace of \( \mathbb R^n) \). 
		Let \( \bar \sigma\co \mathbb B^{n-1} \times \mathbb R \to \mathbb B^{n-1} \times \mathbb R \) be defined by
			\[
				\bar\sigma( \mathbf x, t ) = \left\{
					\begin{array}{ll}
						(\mathbf x, t+2) & |\mathbf x| \leq \frac12 \\
						(\mathbf x, t+ 2(1-|\mathbf x|)) & |\mathbf x| > \frac12
					\end{array}\right.
			\]
		Define \( \sigma = \vp^{-1}\circ \bar \sigma \circ \vp \).
		Observe that \( \sigma \) extends continuously to the boundary of \( C \), where it is the identity.
		We can therefore extend \( \sigma \) to all of \( M \) by the identity.
		Observe that \( \sigma \) satisfies conditions (i) and (ii).
		As \( C \) is a standard open ball, we can run the same argument above with \( B \) replaced by \( C \) to get a homeomorphism \( \tau \) satisfying conditions (iii) and (iv), as \( \sigma \) is supported in the closure of \( C \). 
		The homeomorphisms \( \sigma \) and \( \tau \) are visualized in \Cref{fig:manifolds}.
	\end{proof}
	Let us recall the fragmentation lemma, and then we can show that \( \Homeo_0(M) \) is SB-generated. 

	\begin{theorem}[Fragmentation Lemma]\label{thm:fragmentation}
		Let \( M \) be a closed manifold.
		If \( B_1, \ldots, B_n \) is an open covering of \( M \), then there exists an open neighborhood \( U \) of the identity in \( \Homeo_0(M) \) such that for each \( g \in U \) there exists \( g_1, \ldots, g_n \in \Homeo_0(M) \) satisfying \( g= g_1 \circ \cdots \circ g_n \) and \( \supp(g_i) \subset B_i \).
		\qed
	\end{theorem}

	\begin{theorem}
		\label{thm:closed manifold}
		If \( M \) is a closed manifold, then \( \Homeo_0(M) \) is SB-generated.
	\end{theorem}

	\begin{proof}
		Fix an open covering of \( M \) by standard open balls \( B_1, \ldots, B_n \), and let \( U \) the associated open neighborhood of the identity given by the Fragmentation Lemma. 
		First note that as \( \Homeo_0(M) \) is connected, \( U \) generates \( \Homeo_0(M) \). 
		It is left to show that \( U \) is strongly bounded in \( \Homeo_0(M) \), which we accomplish by showing that it is strongly distorted. 
		
		Let \( H_i \) be the subgroup of \( \Homeo_0(M) \) consisting of homeomorphisms supported in \( B_i \). 
		Then, by \Cref{lem:ball-distortion}, \( H_i \) is strongly distorted in \( \Homeo_0(M) \). 
		\Cref{lem:product of distorted is distorted} implies that \( H_1H_2\cdots H_n \) is strongly distorted in \( \Homeo_0(M) \), and hence, as \( U \subset H_1\cdots H_n \), it is also strongly distorted. 
		Therefore, by \Cref{lem:distorted implies bounded}, \( \Homeo_0(M) \) is SB-generated.
	\end{proof}

\section{Big mapping class groups}
\label{sec:mcg}

The \emph{mapping class group} of a surface \( M \) is the group of isotopy classes of homeomorphisms \( M \to M \). 
A borderless surface is of \emph{finite type} if it is homeomorphic to the interior of a compact surface; otherwise, it is of \emph{infinite type}.
Mapping class groups of infinite-type surfaces are referred to as \emph{big mapping class groups} (see \cite{AramayonaBiga} for an introductory survey).
The goal is to exhibit a family of big mapping class groups that are SB-generated.
To do so, we need to introduce the notion of a telescoping surface.

\begin{Def}
	A surface \( M \) is \emph{telescoping} if it admits two separating simple closed curves \( a \) and \( b \) and a topological end \( \mu \) such that the following hold:
	\begin{enumerate}[(i)]
		\item \( a \) separates \( b \) from \( \mu \) (i.e., any proper ray with base point on \( b \) and exiting the end \( \mu \) must intersect \( a \)),
		\item given any separating simple closed curve \( c \) separating \( a \) from \( \mu \), there exists a homeomorphism that fixes \( b \) and maps \( a \) onto \( c \),
		\item if \( T \) is the component of \( M \ssm a \) disjoint from \( b \), then there exists a homeomorphism \( M \to M \) mapping \( M \ssm T \) into \( T \). 
	\end{enumerate}
	The set \( T \) is called a \emph{maximal telescope} and the end \( \mu \) is called a \emph{maximal end}.
\end{Def}

The definition of telescoping above is given in \cite{VlamisHomeomorphism}; it is an expansion of the original definition given by Mann--Rafi in \cite{MannLargescalea}.
The prototypes for telescoping surfaces are the 2-sphere, the plane, and the open annulus (the definition, as given above, does not include the 2-sphere, but we should view the 2-sphere as a telescoping surface).
The next simplest example is obtained by removing a Cantor set from the 2-sphere; in this case, each end is maximal. 

Let us describe two other examples. 
Let \( X_1 \) be the Cantor set, and let \( X_2 \) be a compact zero-dimensional Hausdorff space that can be written as \( X_2 = P \cup D \), where \( P \) is a non-trivial perfect set and \( D \) is a discrete set accumulating onto \( P \) (i.e., \( \overline D \ssm D = P \)). 
Fix an accumulation point \( x_i \) in \( X_i \). 
Let \( Z = X_1 \sqcup X_2/ \sim \), where \( \sim \) is the equivalence relation generated by declaring \( x_1 \sim x_2 \). 
The surface \( \Sigma \) resulting from removing an embedded copy of \( Z \) from the 2-sphere is telescoping, and it has a unique maximal end corresponding to the equivalence class of \( x_1 \) (or equivalently, \( x_2 \)). 
Moreover, \( \Sigma \# \Sigma \) is also a telescoping surface, but this time with two maximal ends\footnote{The surface \( \Sigma \) is not considered telescoping in the Mann--Rafi definition, whereas \( \Sigma\#\Sigma \) is.}. 

The set of maximal ends of a telescoping surface \( M \) consists of a single \( \Homeo(M) \)-orbit and is either a singleton, a doubleton, or a perfect set (see \cite{VlamisHomeomorphism} for details). 
A telescoping surface with a perfect set of maximal ends is referred to as a \emph{perfectly self-similar surface}, and we call those with a unique maximal end \emph{uniquely telescoping}.

% To state \Cref{thm:telescoping manifolds}, we need the notion of a tame surface, which itself requires several other definitions.
% We view the end space of a surface as a triple \( (E, E_{\text{np}}, E_{\text{no}}) \), where \( E_{\text{no}} \subset E_{\text{np}} \subset E \) with \( E_{\text{np}} \) and \( E_{\text{no}} \) closed and consisting of the non-planar and non-orientable ends, respectively. 
% When we say two subsets \( U \) and \( V \) of \( E \) are homeomorphic, we mean that there is a homeomorphism \( U \to V \) mapping \( U\cap E_{\text{np}} \) onto \( V \cap E_{\text{np}} \) and \( U \cap E_{\text{no}} \) onto \( V \cap E_{\text{no}} \). 
% An end is \emph{stable} if it admits a neighborhood basis consisting of pairwise-homeomorphic clopen subsets.
%
% We state the remaining definitions in terms of telescoping surfaces, where they are simpler.
% In a telescoping surface, an end \( e \) is an \emph{immediate predecessor} of the maximal ends if every end in the closure of the \( \Homeo(S) \)-orbit of \( e \) is either in the orbit itself or is maximal. 
% A telescoping surface is \emph{tame} if each immediate predecessor of the maximal ends is stable. 

Finally, recall from the introduction that given a subsurface \( \Sigma \) of a surface \( S \), we let \( U_\Sigma \) denote the subset of \( \mcg(S) \) consisting of mapping classes that admit a representative homeomorphism that restricts to the identity on the complement of \( \Sigma \).

\begin{theorem} \label{thm:telescoping manifolds}
	Let \( S \) be a surface that can be expressed as the connected sum of a finite-type borderless surface and finitely many telescoping surfaces \( M_1, \ldots, M_n \).
	If, for each \( M_i \) that is uniquely telescoping, there exists \( j \neq i \) with \( M_j \) homeomorphic to \( M_i \), then there exists a finite-type subsurface \( \Sigma \) of \( S \) such that \( U_\Sigma \) is strongly distorted in \( \mcg(S) \).
\end{theorem}

% \begin{theorem} \label{thm:telescoping manifolds}
% 	Let \( S \) be a surface that can be expressed as the connected sum of a finite-type borderless surface and finitely many tame telescoping surfaces \( M_1, \ldots, M_n \).
% 	If, for each \( M_i \) that is uniquely telescoping, there exists \( j \neq i \) with \( M_j \) homeomorphic to \( M_i \), then \( \mcg(S) \) is SB-generated.
% \end{theorem}

\begin{proof}
	Throughout the proof, we rely heavily on the results of \cite{MannLargescalea} and \cite{VlamisHomeomorphism}.
	If \( M_i \) is perfectly self-similar, then \( M_i \) is homeomorphic to \( M_i \# M_i \), and if \( M_i \) has exactly two maximal ends, then \( M_i \) is homeomorphic to \( M_i' \# M_i' \) with \( M_i' \) uniquely telescoping.
	We may therefore assume that none of the \( M_i \) have exactly two maximal ends and that for every \( i \) there exists \( j \neq i \) such that \( M_i \) and \( M_j \) are homeomorphic.  
	These assumptions simplify the notation below.  

	We can choose pairwise-disjoint subsurfaces \( T_1, \ldots, T_n \) such that the closure of the complement of \( T_1 \cup \cdots \cup T_n \) is a finite-type surface \( \Sigma \) with \( n \) boundary components and such that \( T_i \) is homeomorphic to \( M_i \) with an open disk removed.  
	Let \( U = U_\Sigma \).
	The rest of the proof is dedicated to showing that \( U \) is strongly distorted in \( \mcg(S) \). 
	% The subset \( U = U_\Sigma \) consisting of mapping classes with a representative restricting to the identity on \( \Sigma \) is an open identity neighborhood in \( \mcg(S) \). 
	% In \cite[Section~6]{MannLargescalea}, Mann--Rafi construct a finite set \( A \) such that \( U \cup A \) is a generating set for \( \mcg(S) \).  
	% Therefore, if we show that \( U \) is strongly distorted, and hence strongly bounded by \Cref{lem:distorted implies bounded}, we will have established that \( U \cup A \) is a strongly bounded generating set for \( \mcg(S) \). 

	For each \( g \in U \), there exists \( g_1, \ldots, g_n \in U \) such that \( g = g_1 \circ \cdots \circ g_n \) and such that \( g_k \) has a representative supported in \( T_k \). 
	Let \( V_k \) be the subset of \( \mcg(S) \) consisting of mapping classes with a representative supported in \( T_k \). 
	We claim that \( V_k \) is strongly distorted, and hence as \( U = V_1V_2 \cdots V_n \), \Cref{lem:product of distorted is distorted} implies that \( U \) is strongly distorted.
	
	Fix \( k \in \{1, \ldots, n \} \).
	There are two cases: either \( M_k \) is uniquely telescoping or perfectly self-similar. 
	First, let us assume \( M_k \) is uniquely telescoping.
	Fix an embedding \( \iota_k \co T_k \hookrightarrow M_k \) so that \( M_k \ssm \iota_k(T_k) \) is an open disk. 
	By \cite[Proposition~6.18]{MannLargescalea}, there is an element \( f \in A \) such that
	\begin{itemize}
		\item \( f(T_k) \subset T_k \),
		\item \( \iota_k(f(T_k)) \) is a maximal telescope in \( M_k \) (i.e., its boundary can be taken to be the curve \( a \) in the definition of telescoping), and
		\item \( \iota_k(T_k) \) is---in the language of \cite{VlamisHomeomorphism}---an \emph{extension} of the telescope \( \iota_k(f(T_k)) \) (i.e., its boundary can be taken to be the curve \( b \) in the definition of telescoping).
	\end{itemize}
	Let \( H_1 \) and \( H_2 \) be the subgroups of \( \mcg(M_k) \) consisting of mapping classes supported in \( \iota_k(f(T_k)) \) and \( \iota_k(T_k) \), respectively.
	Restricting the proof of \cite[Theorem~5.2]{VlamisHomeomorphism} to telescopes\footnote{The proof is actually about telescopes, as the first step is to fragment into homeomorphisms supported in telescopes.}, rather than telescoping surfaces, implies that \( H_1 \) is a strongly distorted in \( H_2 \).
	As \( \iota_k \) induces an isomorphism \( V_k \to H_2 \) mapping \( fV_kf^{-1} \) onto \( H_1 \), we have that \( fV_kf^{-1} \) is strongly distorted in \( V_k \).
	Therefore, \( fV_kf^{-1} \) and hence \( V_k \) is strongly distorted in \( \mcg(S) \), as desired. 

	Now, assume that \( M_k \) is perfectly self-similar. 
	The structure of perfectly self-similar surfaces is detailed in \cite{VlamisHomeomorphisma}, and we refer the reader there for details of the facts we use below.
	The main takeaway is that perfectly self-similar surfaces all behave like the sphere minus the Cantor set.
	Therefore, the reader unfamiliar with perfectly self-similar surfaces may find it helpful in the arguments below to imagine \( M_k \) being the 2-sphere with a Cantor set removed and \( T_k \) being the closed disk with a Cantor set removed from its interior.

	The perfectly self-similar case is similar in spirit to the closed manifold case, and in fact, we will simply describe how \Cref{fig:manifolds} also applies to the case at hand. 
	By assumption, there exists \( j \neq k \) such that \( M_j \) is homeomorphic to \( M_k \). 
	Let \( D \) be a subsurface of \( S \) such that \( T_j\cup T_k \subset D \) and \( D \ssm T_j \cup T_k \) is a pair of pants. 
	Note that \( D \) is  homeomorphic to \( T_k \).
	In particular, capping off the boundary component of \( D \) (respectively, \( T_k \) or \( T_j \)) with a disk results in a telescoping surface, namely one homeomorphic to \( M_k \), and hence we can talk about maximal ends of \( D \) (respectively, \( T_k \) and \( T_j \)).

	Fix two maximal ends of \( D \) not seen by \( T_k \) and label them \( \mu_{\pm\infty} \).
	We can then find pairwise-disjoint  subsurfaces \( \{ C_n \}_{n\in\bz} \) such that:
	\begin{itemize}
		\item each \( C_n \) sees a maximal end of \( D \),
		\item each end of \( D \), other than \( \mu_{\pm\infty} \), is seen by one of the \( C_n \),
		\item \( T_k \) is contained in the interior of \( C_0 \),
		\item there is a maximal end of \( D \) seen by \( C_0 \) but not seen by \( T_k \), and
		\item for each neighborhood \( W \) in the surface of \( \mu_{\pm\infty} \) there exists \( N \in \mathbb N \) such that \( C_{\pm n} \subset W \) for all \( n > N \). 
	\end{itemize}
	Self-similarity together with the first condition guarantee that the \( C_n \) are pairwise homeomorphic, and in particular, each \( C_n \) is homeomorphic to \( D \). 

	Choose a homeomorphism \( \vp \colon D \to C_0 \) such that \( \vp(C_0) = T_k \).
	Let \( B_n = \vp(C_n) \). 
	We can now choose a homeomorphism \( \sigma \) supported in \( C_0 \) such that \( \sigma(B_n) = B_{n+1} \). 
	Let \( \tau = \vp^{-1}\circ \sigma \circ \vp \), so that \( \tau(C_n) = C_{n+1} \) and \( \tau \) is supported in \( D \).
	Setting \( C = C_0 \) and \( B = T_k \), the setup we just described is exactly as shown in \Cref{fig:manifolds}. 
	Now, in \Cref{lem:construction}, the requirement that the diameters tend to zero in the second and third conditions can be replaced with the requirement that, under iteration, the sets leave every compact set.
	This is discussed in \cite{LeRouxStrong} and is directly deduced from the original statement by considering a metric on  \( S \) coming from the restriction of a metric on the Freudenthal compactification of \( S \).
	As a result, we can apply \Cref{prop:construction} to see that the set of homeomorphisms supported in \( T_k \) is a strongly distorted group, and hence so is \( V_k \), being a quotient of a strongly distorted group. 
%	In particular, \( V_k \) is strongly bounded in \( \mcg(S) \). 

	We have shown that each of the \( V_k \) is strongly distorted in \( \mcg(S) \), and therefore that \( U = V_1V_2\cdots V_n \) is strongly distorted as well. 
	% This in turn established a strongly bounded generating set for \( \mcg(S) \), namely \( U \cup A \), finishing the proof. 
\end{proof}

%-----------------
% Ordinals
%-----------------

\section{Homeomorphism groups of well-ordered spaces}
\label{sec:well-ordered}

We now prove that the homeomorphism group of a compact well-ordered space is SB-generated whenever the limit capacity is a successor ordinal.
Before doing so, we briefly introduce well-ordered spaces and some basic properties. 
For more detailed introduction and proofs of the properties mentioned below, we refer the reader to \cite[Section~2]{BhatAlgebraic}. 

Recall that a binary relation \( \leq \) on a set \( X \) is a \emph{well-order} if it is a total order (i.e., reflexive, antisymmetric, transitive, and strongly connected) and every non-empty subset has a least element with respect to the ordering; the pair \( (X, \leq) \) is called a \emph{well-ordered set}.
Two well-ordered sets \( X \) and \( Y \) are \emph{order isomorphic} if there exists a bijection \( f \co X \to Y \) such that \( x_1 \leq x_2 \) if and only if \( f(x_1) \leq f(x_2) \) for all \( x_1, x_2 \in X \).  

The \emph{order topology} on the well-ordered set \( X \) is generated by sets of the form \( \{x : a < x\} \) and \( \{ x: x < b\} \) for all \( a,b \in X \), where \( < \) is the strict ordering induced by \( \leq \). 
A \emph{well-ordered space} is a well-ordered set equipped with its order topology.
Every well-ordered space is Hausdorff and zero-dimensional (i.e., it admits a basis of clopen sets).
Observe that an order isomorphism between two well-ordered sets induces a homeomorphism of the corresponding well-ordered spaces.
However, two homeomorphic well-ordered spaces need not be order-isomorphic as sets; in particular, the classification of well-ordered spaces up to homeomorphism is coarser than the classification up to order isomorphism.

Compact well-ordered spaces are classified up to homeomorphism by two invariants: their Cantor--Bendixson rank, which is ordinal valued, and the Cantor--Bendixson degree, which is a natural number.
The rank is always a successor ordinal, and the predecessor of the rank is called the \emph{limit capacity}.

There is a natural connection between big mapping class groups and homeomorphism groups of well-ordered spaces. 
Given a manifold \( M \) with end space \( E \), the action of \( \Homeo(M) \) on \( M \) induces an action of \( \Homeo(M) \) on \( E \), and this action factors through the mapping class group, yielding a homomorphism \( \mcg(M) \to \Homeo(E) \).
Let \( X \) be a countable compact well-ordered space.
Then \( X \) can be embedded in the 2-sphere; let \( M_X \) be the surface  obtained by removing an embedded copy of \( X \) in the 2-sphere.
The end space of \( M_X \) is homeomorphic to \( X \), yielding a homomorphism \( \mcg(M_X) \to \Homeo(X) \). 
By \cite{RichardsClassification}, this homomorphism is surjective, allowing us to view \( \mcg(M_X) \) as a type of braid group over \( X \). 
This is a fruitful picture to the geometric topologist, and it is how the author views the elements of \( \Homeo(X) \).
We note that this picture does not hold if \( X \) is uncountable, as it can no longer be embedded in \( \mathbb R^2 \), but nonetheless, the intuition holds. 

Given this relation just described, the work of Mann--Rafi \cite{MannLargescalea} on big mapping class groups implies that \( \Homeo(X) \), equipped with the compact-open topology, is CB-generated when \( X \) is a countable compact well-ordered space with successor limit capacity. 
Their arguments are in a setting more general than well-ordered spaces, and so we do not use their proof directly but rather as inspiration.

Before continuing, we recall several facts about well-ordered spaces that we will implicitly use throughout our argument. 
Let us first recall the definition of the Cantor--Bendixson rank and degree of a space.
The rank is defined via transfinite recursion. 
Let \( X \) be a topological space, set \( X^{(0)} = X \), and let \( X' \) denote the set of accumulation points in \( X \). 
For an ordinal \( \alpha \), set  \( X^{(\alpha+1)} \) to be \( (X^{(\alpha)})' \).
For a limit ordinal \( \lambda \), set \( X^{(\lambda)} = \bigcap_{\alpha < \lambda} X^{(\alpha)} \).
The set \( X^{(\alpha)} \) is called the \emph{\( \alpha^{\text{th}} \) Cantor--Bendixson derivative of \( X \)}.
These derivatives eventually stabilize, and the least ordinal \( \alpha \) satisfying \( X^{(\alpha+1)} = X^{(\alpha)} \) is the \emph{Cantor--Bendixson rank of \( X \)}. 
If \( X \) is compact and if its Cantor--Bendixson derivatives are eventually empty, then the Cantor--Bendixson rank is a successor ordinal and the last nonempty derivative is finite; the cardinality of this set is called the \emph{Cantor--Bendixson degree of \( X \)}. 

In a compact well-ordered space \( X \), the Cantor--Bendixson derivatives are eventually empty, allowing us to define the Cantor--Bendixson rank and degree, which we denote \( \mathrm{rank}(X) \) and \( \mathrm{deg}(X) \), respectively. 
In this notation, the classification above says that two compact well-ordered spaces \( X \) and \( Y \) are homeomorphic if and only if \( \rank(X) = \rank(Y) \) and \( \deg(X) = \deg(Y) \). 

Let \( X \) be a compact well-ordered space.
If \( A \) is a compact subset of \( X \), then with respect to the subspace topology, it is a compact well-ordered space with \( \rank(A) \leq \rank(X) \).
If \( Y \) is a compact well-ordered space with \( \rank(Y) < \rank(X) \), then there exists a clopen subspace \( A \) of \( X \) homeomorphic to \( Y \). 
For \( x \in X \), we define \( \rank(x) \) to be the rank of the set \( \{y \in X: y \leq x \} \). 
Note that the rank of a point is always a successor ordinal. 
There exists a neighborhood basis for \( x \) consisting of clopen sets with Cantor--Bendixson rank equal to \( \rank(x) \) and of Cantor--Bendixson degree one; in particular, these neighborhood basis elements are pairwise homeomorphic.

With the basic properties given above, we can turn to proving \Cref{thm:well-ordered}.

\begin{theorem} 
\label{thm:well-ordered}
	If the limit capacity of a compact well-ordered space is a successor ordinal, then its homeomorphism group is SB-generated.
\end{theorem}

\begin{proof}
	Let \( X \) be a compact well-ordered space whose limit capacity is a successor ordinal.
	Set \( d = \deg(X) \). 
	If \( d = 1 \), then \( \Homeo(X) \) is strongly bounded \cite[Theorem~3.14]{BhatAlgebraic}, and hence SB-generated; we may therefore assume that \( d > 1 \). 
	Choose pairwise-disjoint clopen subsets \( \Omega_1, \ldots,\Omega_d \) such that \( \rank(\Omega_k) = d \) and \( X = \bigsqcup_{k=1}^d \Omega_k \).
	Note that each of the \( \Omega_k \) contains exactly one of the \( d \) maximal rank elements of \( X \), call it \( \mu_k \).  

	Let \( U_k \) be the subgroup of \( \Homeo(X) \) consisting of homeomorphisms supported in \( \Omega_k \).
	Observe that each element in \( U_k \) commutes with each element of \( U_j \) whenever \( j \neq k \), allowing us to define the subgroup \( U = U_1U_2\cdots U_k \).
	In fact, \( U \) is simply the subgroup of \( \Homeo(X) \) that  stabilizes each of the \( \Omega_k \) setwise. 
	We first claim that \( U \) is strongly bounded in \( \Homeo(X) \). 
	As every homeomorphism of \( \Omega_k \) can be extended to all of \( X \) by the identity, we see that \( U_k \) is homeomorphic to \( \Homeo(\Omega_k) \). 
	Now, \( \Omega_k \) is a well-ordered space with Cantor--Bendixson degree one and whose limit capacity is a successor ordinal; therefore, by \cite[Theorem~3.14]{BhatAlgebraic}, \( \Homeo(\Omega_k) \)---and hence \( U_k \)---is strongly distorted. 
	As \( U \) is the product of the \( U_k \), \Cref{lem:product of distorted is distorted} implies \( U \) is strongly distorted and hence strongly bounded by \Cref{lem:distorted implies bounded}. 
	In particular, \( U \) as a strongly bounded group, is strongly bounded as a subset of \( \Homeo(X) \). 

	Let \( G \) be the subgroup of \( \Homeo(X) \) that stabilizes each of the \( \mu_k \), and note that \( U < G \).
	We claim that that there exists a finite set \( F \subset G \) such that \( U \cup F \) generates \( G \), implying that \( G \) is SB-generated, as \( U \cup F \) is strongly bounded in \( G \).  
	As \( G \) is finite index in \( \Homeo(X) \), it will follow that \( \Homeo(X) \) is SB-generated as well.

	Let \( \alpha \) be the ordinal satisfying \( \rank(X) = \alpha + 2 \); such an ordinal exists as the limit capacity is a successor ordinal, implying by definition that the Cantor--Bendixson rank is the successor of a successor. 
	To get our finite set \( F \), for \( 1 \leq j < k \leq d \), we are going to take a pair of homeomorphisms, each of which shifts elements of rank \( \alpha + 1 \) away from \( \mu_j \) and towards \( \mu_k \). 

	Using the fact that the limit capacity is a successor ordinal, we can write \( \Omega_k \ssm \{\mu_k\} = \bigcup_{n \in \mathbb N} A_{k,n}
	\), where the \( A_{k,n} \) are pairwise-disjoint pairwise-homeomorphic clopen sets of Cantor--Bendixson rank \( \alpha +1 \) and of degree one.
	This can readily be accomplished as follows:
	as there are countably many rank \( \alpha + 1 \) elements in \( \Omega_k \), enumerate them from least to greatest, say \( x_1 < x_2 < \cdots \), and set \( x_0 = \min \Omega_k \).
	We can then define \( A_{k,n} = \{ y \in \Omega_k: x_{n-1} < y \leq x_n \} \). 

	For \( j,k \in \{1, \ldots, d\} \) with \( j < k \), choose a homeomorphism \( e_{j,k} \in G \) that satisfies the following:
	\begin{itemize}
		\item \( e_{j,k}(A_{k,2n}) = A_{k,2n+2} \) for all \( n \in \mathbb N \),
		\item \( e_{j,k}(A_{j,2}) = A_{k,2} \),
		\item \( e_{j,k}(A_{j,2n}) = A_{j,2n-2} \) for all \( n \in \mathbb N\ssm\{1\} \), and
		\item \( e_{j,k}(x) = x \) for all other \( x \in X \).
	\end{itemize}
	Similarly, choose \( o_{j,k} \in G \) satisfying:
	\begin{itemize}
		\item \( o_{j,k}(A_{k,2n-1}) = A_{k,2n+1} \) for all \( n \in \mathbb N \),
		\item \( o_{j,k}(A_{j,1}) = A_{k,1} \),
		\item \( o_{j,k}(A_{j,2n-1}) = A_{j,2n-3} \) for all \( n \in \mathbb N \ssm\{1\} \), and
		\item \( o_{j,k}(x) = x \) for all other \( x \in X \). 
	\end{itemize} 
	The existence of the \( e_{j,k} \) and \( o_{j,k} \) is readily deduced from the classification of compact well-ordered spaces, and the fact that the \( A_{k,n} \) are clopen.  
	Observe that both \( e_{j,k} \) and \( o_{j,k} \) are shifting points from \( \Omega_j \) to \( \Omega_k \), but a disjoint set of points; moreover, they each shift a unique rank \( \alpha+1 \) point from \( \Omega_j \) to \( \Omega_k \).
	Let \[ F = \left\{ e_{j,k}, o_{j,k}, e_{j,k}^{-1}, o_{j,k}^{-1}: j,k \in \{1, \ldots, d\}, j < k \right\}. \] 
	We claim \( U \cup F \) generates \( G \). 

	Given
	\begin{itemize}
		\item \( g \in G \)
		\item \( j,k \in \{ 1, \ldots, d\} \) with \( j < k \), and
		\item an ordinal \( \beta \) with \( \beta \leq \alpha \),
	\end{itemize}
	let \( O_{j,k,\beta}(g) \) denote the set of rank \( \beta + 1 \) points in \( \Omega_j \) that \( g \) maps into \( \Omega_k \). 
	Similarly, let \( I_{j,k,\beta}(g) \) denote the set of the rank \( \beta + 1 \) points in \( \Omega_k \) that \( g \) maps into \( \Omega_j \). 
	By definition, \( g \in U \) if and only if \( I_{j,k,\beta}(g) = O_{j,k,\beta}(g) = \varnothing \) for all  \( j,k \in \{1, \ldots, d\} \) with \( j < k \) and for all \( \beta \). 
	As \( g \) stabilizes each of the \( \mu_k \), the cardinality of \( O_{j,k,\alpha}(g) \) and of \( I_{j,k,\alpha}(g) \) is finite.
	Using that every set of ordinals is well-ordered, this allows to define the following quantity: 	
		\[
			\lambda_{j,k}(g) := \min\{ \beta : |O_{j,k,\beta}(g)|, |I_{j,k,\beta}(g)| < \infty \}.	
		\]

	Fix \( g \in G \).
	The goal of what follows is to construct an element \( h_{j,k} \) for each \( j < k \) in the subgroup generated by \( U \cup F \) such that \( O_{j,k,\beta}(h_{j,k}\circ g) = I_{j,k,\beta}(h_{j,k}\circ g) = \varnothing  \) for all \( \beta \).
	We construct \( h_{j,k} \) in steps, with each step reducing the value of \( \lambda_{j,k} \). 
	To simplify notation, fix  \( j \) and \( k \) in \( \{1, \ldots, d \} \) such that \( j < k \), and set \( \lambda = \lambda_{j,k} \), \( O_\beta = O_{j,k,\beta}(g) \), \( I_\beta = I_{j,k,\beta}(g) \), \( e = e_{j,k} \), and \( o = o_{j,k} \).

	If \( \lambda(g) < \alpha \), set \( g_0 \) to be the identity. 
	Otherwise, \( \lambda(g) = \alpha \), and there exist \( a,b \in G \) such that \( a \) and \( b \) are supported in \( \Omega_k \) and \( \Omega_j \), respectively, and such that
		\begin{itemize}
			\item \( (a\circ g)(O_\alpha) \subset \bigcup_{i=1}^{|O|} A_{k,2i} \), and
			\item \( (b\circ g)(I_\alpha) \subset \bigcup_{i=1}^{|I|} A_{j,{2i-1}} \). 
		\end{itemize}	 
	Let \( g_0 = o^{|I_\alpha|} \circ e^{-|O_\alpha|} \circ a \circ b \in (U\cup F)^{1+|O_\alpha|+|I_\alpha|} \). 
	It follows that \( \lambda_0 := \lambda( g_0 \circ g ) < \alpha \). 

	Now, there exist \( c, d \in G \) such that \( c \) and \( d \) are supported in \( \Omega_k \) and \( \Omega_j \), respectively, and such that 
	\begin{enumerate}[(a)]
		\item \( (c\circ g_0 \circ g )(O_{\lambda_0}) \subset A_{k,2} \), and
		\item \( (d \circ g_0 \circ g )(I_{\lambda_0}) \subset A_{j,1} \).
	\end{enumerate}
	If \( O_{\lambda_0} \) (resp., \( I_{\lambda_0} \)) is empty, we simply choose \( c \) (resp., \( d \)) to be the identity. 
	Let \[ \widehat O_2 = \{ x \in A_{k,2} : x \leq \max [(c\circ g)(O_{\lambda_0})] \} \] and let \[ \widehat I_1 = \{ x \in A_{k,1} : x \leq \max [(o \circ d \circ g)(I_{\lambda_0})] \}. \]
	Our goal is to simultaneously map \( \widehat O_2 \) and \( A_{k,1} \ssm \widehat I_1 \) into \( \Omega_j \) while moving no other elements of \( \Omega_k \) out of  \( \Omega_k \), thereby decreasing the value of \( \lambda \).  

	Choose clopen subsets \( \widehat I_2 \subset A_{k,2} \ssm \widehat O_2 \) and \( \widehat O_1 \subset A_{k,1}\ssm \widehat I_1 \) such that \( \widehat I_2 \) is homeomorphic to \( \widehat I_1 \) and \( \widehat O_1 \) is homeomorphic to \( \widehat O_2 \). 
	For \( i\in \{1,2\} \), let \( \widehat A_{k,i} = A_{k,i} \ssm (\widehat O_i \cup \widehat I_i) \), so that \[ A_{k,i} = \widehat A_{k,i} \sqcup \widehat O_i \sqcup \widehat I_i. \]  
	Now, choose homeomorphisms
	\begin{align*}
		h_1 \co & \widehat A_{k,1} \to \widehat A_{k,1} \sqcup \widehat I_1 \\
		h_2 \co & \widehat A_{k,2} \to \widehat A_{k,2} \sqcup \widehat O_2 \\
		h_3 \co & \widehat I_1 \to \widehat I_2 \\
		h_4 \co & \widehat O_2 \to \widehat O_1  
	\end{align*}
	and define the homeomorphism \( h \co A_{k,1} \sqcup A_{k,2} \to A_{k,1}\sqcup A_{k,2} \) by \[ h = h_1 \sqcup h_2 \sqcup h_3 \sqcup h_4. \]
	Extending \( h \) by the identity to the rest of \( X \), we may view \( h \) as an element of \( U_k \). 
	Setting \( g_1 = o^{-1}\circ h \circ o \circ d \circ c \in (U\cup F)^4 \), we have by construction that \[ \lambda_1 := \lambda(g_1  \circ g_0 \circ g)< \lambda_0. \]

	Repeating this process, we construct \( g_1, g_2, \ldots \) in \( (U\cup F)^4 \) and a decreasing sequence of ordinals \( \lambda_1 > \lambda_2 > \cdots \), with \( \lambda_m = \lambda(g_m\circ \cdots \circ g_1 \circ g_0 \circ g) \). 
	As every decreasing sequence of ordinals is finite, this process stops in finitely many steps, say \( M \) steps.
	All together, \[ h_{j,k} := g_M\circ \cdots g_1 \circ g_0 \in (U\cup F)^{4M+1+|O_{j,k,\alpha}(g)|+|I_{j,k,\alpha}(g)|} \] satisfies 
	\[ O_{j,k,\beta}(h_{j,k}\circ g) = I_{j,k,\beta}(h_{j,k}\circ g) = \varnothing \]
	for all \( \beta \). 
	In other words, \( h_{j,k}\circ g \) maps no element of \( \Omega_j \) into \( \Omega_k \) and vice versa. 

	Order the pairs of integers \( (j,k) \) with \( j < k \) and \( j,k \leq d \) lexicographically. 
	Recursively perform the above process for each pair \( (j,k) \).
	The end result is an element \( u \) of \( \langle U \cup F \rangle \) such that \( u \circ g \) stabilizes each of the \( \Omega_k \), i.e., \( u \circ g \in U \). 
	Therefore, \( g \) is in the subgroup generated by \( U \cup F \); in particular, \( G = \langle U \cup F \rangle \), as desired. 
\end{proof}

We finish by proving \Cref{cor:proper class}, exhibiting sets of pairwise non-isomorphic SB-generated groups of arbitrary cardinality.

\begin{corollary} 
\label{cor:proper class}
	For any cardinal \( \kappa \), there exists a set of pairwise non-isomorphic, non-finitely generated, non-strongly bounded, SB-generated groups of cardinality \( \kappa \). 
\end{corollary}

\begin{proof}
	Given an ordinal \( \alpha \), let \( X_\alpha \) be the set containing all ordinals less than \( \alpha \).
	Note that for any ordinal \( \alpha \), the sets \( X_{\alpha} \) and \( X_{\alpha+1} \) have the same cardinality. 
	
	Let \( \kappa \) be a cardinal, let \( \alpha \) be a successor ordinal such that \( |X_{\alpha}| = \kappa \), and let \( \beta \) be such that \( \alpha = \beta + 1 \).
	Let \( A \) be the subset of \( X_\alpha \) consisting of all successor ordinals in \( X_\alpha \).  
	By definition, \( |A| \leq \kappa \). 
	Now, define \( f \co X_\beta \to X_{\alpha} \) by \( f(\gamma) = \gamma + 1 \). 
	By definition, \( f \) is injective and the image of \( f \) is \( A \), implying \( |A|\geq |X_\beta| = \kappa \). 
	Therefore, \( |A| = \kappa \). 

	For each \( \gamma \in A \), let \( Y_\gamma \) be a well-ordered space whose Cantor--Bendixson rank is \( \gamma +1 \) and whose degree is two (for instance, letting \( \eta = \omega^\gamma\cdot 2 + 1 \), take \( Y_\gamma = X_{\eta} \)). 
	Now, let \( S = \{ \Homeo(Y_\gamma) : \gamma \in A \} \). 
	By \cite[Theorem~29]{GheysensDynamics}, \( \Homeo(Y_\gamma) \) is isomorphic to \( \Homeo(Y_{\gamma'}) \) if and only if \( \gamma = \gamma' \). 
	By \Cref{thm:well-ordered}, the group \( \Homeo(Y_\gamma) \) is SB-generated; moreover, by \cite[Corollary~1.3]{BhatAlgebraic}, \( \Homeo(Y_\gamma) \) surjects onto \( \mathbb Z \) and hence is not strongly bounded. 
	Therefore, \( S \) is a set of cardinality \( \kappa \) consisting of pairwise non-isomorphic non-strongly bounded non-finitely generated SB-generated groups. 
\end{proof}

%------------
% Bibliography
%------------

\bibliographystyle{amsplain}
%\bibliography{sb-groups-bib}
\bibliography{Zotero}

\end{document}